\documentclass[11pt,oneside,reqno]{amsart}

\usepackage{amsmath,amssymb,amsthm,textcomp,mathtools}
\usepackage{amsfonts,graphicx}
\usepackage[mathscr]{eucal}
\usepackage{color}
\usepackage{csquotes}
\usepackage{enumitem}
\allowdisplaybreaks

\numberwithin{equation}{section}

\theoremstyle{definition}

\numberwithin{equation}{section}

\newtheorem{theorem}{\bf Theorem}[section]
\newtheorem{remark}{\bf Remark}[section]
\newtheorem{proposition}{Proposition}[section]

\newtheoremstyle
{remarkstyle}
{}
{11pt}
{}
{}
{\bfseries}
{:}
{     }
{\thmname{#1} \thmnumber{#2} }

\theoremstyle{remarkstyle}

\begin{document}
	\title{Skellam Processes via Multiparameter Poisson Process}
	\author[Pradeep Vishwakarma]{Pradeep Vishwakarma}
	\address{Pradeep Vishwakarma, Theoretical Statistics and Mathematics Unit,
		Indian Statistical Institute, Kolkata, 700108, India.}
	\email{vishwakarmapr.rs@gmail.com}
	
	
	\date{\today}	
	\maketitle
\begin{abstract} 
We introduce a mltiparameter version of Skellam point process via multiparameter Poisson processes. Its distributional properties are studied in detail. Its compound representation is derived for a particular case. Also, its Riemann integral over a rectangle in $\mathbb{R}^M_+$, $M\ge1$ is introduced and a closed expression for its characteristic function is obtained. Later, we introduce a different version of multiparameter Skellam process, and derive a weak convergence result for it. Moreover, a two parameter fractional Skellam process is discussed.
\end{abstract}
	
\section{Introduction}\label{sec1}
The Skellam type distribution was first introduced and examined in \cite{Irwin1937, Skellam1946}. Since then, the Skellam process has been extensively studied, both from a theoretical perspective and for its diverse practical applications. It appears in various fields ranging from insurance to image processing, finance (see \cite{Barndorff-Nielsen2011}), \textit{etc}.  Skellam processes defined using a generalized counting process, together with their time-changed variants are studied in \cite{Kataria2022}. For a more detailed discussion on the difference of two independent Poisson processes, we refer the reader to \cite{Barndorff-Nielsen2011, Carr2011}.
 Additionally, a broader generalization of the Skellam point process and its fractional variants are introduced and studied in \cite{Cinque2025}. 
 
  Recently, a generalized Skellam fields on finite dimensional Euclidean space was introduced and studied in \cite{Vishwakarma2025b}. Also, a two parameter Skellam field with rectangular increments is considered, and its three different fractional variants are studied in detail. In this paper, we introduce and study Skellam point processes defined via a family of multiparameter Poisson process introduced in \cite{Vishwakarma2025}. Here, we consider multiparameter processes with linear increments (see \cite{Iafrate2024, Pedersen2003, Pedersen2004a}), which is different from the increments considered in \cite{Vishwakarma2025b}.

In \cite{Vishwakarma2025}, a multiparameter Poisson process with linear increments is introduced. Let $\mathbb{R}_+$ denote the set of non-negative real numbers. For $M\ge1$, a collection $\{X(\textbf{t}),\ \textbf{t}\in\mathbb{R}^M_+\}$ of random variables is called multiparameter ($M$-parameter) random process. A non-negative integer valued process $\{N(\textbf{t}),\ \textbf{t}\in\mathbb{R}^M_+\}$ is called multiparameter counting process if $N(\textbf{0})=0$, for any $\textbf{t}$, $N(\textbf{t})\ge0$, and for $\textbf{s}=(s_1,\dots,s_M)\preceq \textbf{t}=(t_1,\dots,t_M)$, that is, $s_i\leq t_i$, $i=1,\dots,M$, we have $N(\textbf{s})\leq N(\textbf{t})$.

A multiparameter counting process $\{N(\textbf{t}),\ \textbf{t}\in\mathbb{R}^M_+\}$ is called multiparameter Poisson process (MPP) with rate parameter $\textbf{0}\prec\Lambda=(\lambda_1,\dots,\lambda_M)\in\mathbb{R}^M_+$ if\\
\noindent (i) for any $\textbf{s}\preceq\textbf{t}$, the increment $N(\textbf{t})-N(\textbf{s})$ has same distribution as $N(\textbf{t}-\textbf{s})$, that is, it has stationary increments;\\
\noindent (ii) for $\textbf{t}^{(r)}$, $r=1,\dots,m$ in $\mathbb{R}^M_+$ such that $\textbf{0}\preceq\textbf{t}^{(1)}\prec\dots\prec\textbf{t}^{(m)}$, the increments $N(\textbf{t}^{(2)})-N(\textbf{t}^{(1)}),\dots,N(\textbf{t}^{(m)})-N(\textbf{t}^{(m-1)})$ are independent of each other, that is, it has independent increments;\\
\noindent (iii) the random variable $N(\textbf{t})$ has Poisson distribution with mean $\Lambda\cdot\textbf{t}$, where $\Lambda\cdot\textbf{t}=\lambda_1t_1+\dots+\lambda_Mt_m$.

The MPP is a multiparameter L\'evy process in the sense of \cite{Pedersen2003, Pedersen2004a}. For a finite subset $\mathcal{J}\subset\mathbb{R}-\{0\}$, let $\{N_j(\textbf{t}),\ \textbf{t}\in\mathbb{R}^M_+\}$ be independent MPPs. We consider a multiparameter process $\{S(\textbf{t}),\ \textbf{t}\in\mathbb{R}^M_+\}$ defined by $S(\textbf{t})\coloneqq\sum_{j\in\mathcal{J}}jN_j(\textbf{t})$. We call it the generalized multiparameter Skellam process (GMSP). When $M=1$, it reduces to a generalized Skellam point process introduced and studied in \cite{Cinque2025}. 

As the GMSP belongs to the class of multiparameter L\'evy processes. From the result of \cite{Barndorff-Nielsen2001}, it follows that is equal in distribution to a sum of independent one parameter L\'evy processes. More precisely, it coincides in distribution with an additive L\'evy process as described in \cite{Khoshnevisan2002}. Also, various applications of such multiparameter process in other areas of mathematics are explored in \cite{Khoshnevisan2002}. Multiparameter L\'evy processes have important applications in fields such as, statistical mechanics, and brain imaging (see \cite{Cao1999}). For more detail on both the theoretical aspects and applications of multiparameter L\'evy processes, we refer to \cite{Barndorff-Nielsen2001, Pedersen2003, Pedersen2004a}.

In the next section, we introduce a multiparameter process, namely, the generalized multiparameter Skellam process. It is defined as the sum of scaled independent multiparameter Poisson processes. We establish a weak convergence result and provide several characterizations of this process. Additionally, we evaluate its integral over a rectangle in finite-dimensional Euclidean space and derive an explicit expression for its characteristic function. In the final section, we introduce a different version of the multiparameter Skellam process and obtain a weak convergence result for it. Finally, we focus on the two-parameter case and analyze its time-changed variant, where the time-changing component is a bivariate process whose marginals are independent inverse stable subordinators.
\section{Generalized multiparameter Skellam process}\label{sec2}
In this section, we introduce a multiparameter Skellam process via MPP.
For a finite subset $\mathcal{J}\subset\mathbb{R}-\{0\}$, let $\{N_j(\textbf{t}),\ \textbf{t}\in\mathbb{R}^M_+\}$, $j\in\mathcal{J}$ be independent multiparameter Poisson processes (for definition see Section \ref{sec1}) with rate parameters $\textbf{0}\prec\Lambda_j=(\lambda_1^{(j)},\dots,\lambda_M^{(j)})\in\mathbb{R}^M_+$, $j\in\mathcal{J}$, respectively. Let us consider a multiparameter process $S=\{S(\textbf{t}),\ \textbf{t}\in\mathbb{R}^M_+\}$ defined as follows:
\begin{equation}\label{gmsp}
	S(\textbf{t})\coloneqq\sum_{j\in\mathcal{J}}jN_j(\textbf{t}),\ \textbf{t}\in\mathbb{R}^M_+.
\end{equation}
We call it the generalized multiparameter Skellam process (GMSP) and refer it as $S\sim GMSP\{\Lambda_j\}_{j\in\mathcal{J}}$. Its mean and variance are given by $\mathbb{E}S(\textbf{t})=\sum_{j\in\mathcal{J}}j\Lambda_j\cdot\textbf{t}$ and $\mathbb{V}\mathrm{ar}S(\textbf{t})=\sum_{j\in\mathcal{J}}j^2\Lambda_j\cdot\textbf{t}$, respectively, where $\boldsymbol{\Lambda}_j\cdot\textbf{t}=\lambda_1^{(j)}t_1+\dots+\lambda_M^{(j)}t_M$ is the standard dot product on $\mathbb{R}^M$. The probability generating function (pgf) of GMSP is
\begin{equation}\label{gmsppgf}
\mathbb{E}u^{S(\textbf{t})}=\exp\bigg(\sum_{j\in\mathcal{J}}\Lambda_j\cdot\textbf{t}(u^j-1)\bigg),\ 0<u\leq1,
\end{equation}
where we used $\mathbb{E}u^{N_j(\textbf{t})}=\exp(\Lambda_j\cdot\textbf{t}(u-1))$, the pgf of MPP.

For $M=1$, the GMSP reduces to the one parameter generalized Skellam process (GSP) defined as follows:
\begin{equation}
	S(t)\coloneqq\sum_{j\in\mathcal{J}}jN_j(t),\ t\ge0,
\end{equation}
where $\{N_j(t),\ t\ge0\}$, $j\in\mathcal{J}$ are independent Poisson processes with rates $\lambda_j$, $j\in\mathcal{J}$, respectively. It was introduced and studied in \cite{Cinque2025}. We denote it as $S\sim GSP\{\lambda_j\}_{j\in\mathcal{J}}$. Moreover, if $\{S_k(t),\ t\ge0\}$, $k=1,2,\dots,M$ are independent GSP then the multiparameter process defined by $\sum_{k=1}^{M}S_k(t_k)$, $(t_1,\dots,t_M)\in\mathbb{R}^M_+$ is a GMSP.
\begin{remark}\label{rem21}
	From Remark 3.3 of \cite{Vishwakarma2025}, it follows that for each MPP $\{N_j(\textbf{t}),\ \textbf{t}\in\mathbb{R}^M_+\}$, $j\in\mathcal{J}$ there exist independent one parameter Poisson processes $\{N_1^{(j)}(t),\ t\ge0\},\dots,\{N_M^{(j)}(t),\ t\ge0\}$ with rates $\lambda_1^{(j)},\dots,\lambda_M^{(j)}$, respectively, such that $N_j(\textbf{t})\overset{d}{=}\sum_{k=1}^{M}N_k^{(j)}(t_k)$, $\textbf{t}=(t_1,\dots,t_M)\in\mathbb{R}^M_+$. Here, $\overset{d}{=}$ denotes the equality in distribution. Hence, GMSP satisfies
$S(\textbf{t})\overset{d}{=}\sum_{j\in\mathcal{J}}j\sum_{k=1}^{M}N_k^j(t_k)=\sum_{k=1}^{M}S_k(t_k)$, $ \textbf{t}=(t_1,\dots,t_M)\in\mathbb{R}^2_+$, where $\{S_k(t),\ t\ge0\}$'s are independent one parameter GSP. Moreover, for $\textbf{s}=(s_1,\dots,s_M)$ and $\textbf{t}=(t_1,\dots,t_M)$ in $\mathbb{R}^M_+$, the auto covariance of MPP is $\mathbb{C}\mathrm{or}(N_j(\textbf{s}),N_j(\textbf{t}))=\sum_{k=1}^{M}\lambda_ks_k\min\{s_k,t_k\}$. Therefore, the auto covariance of GMSP is given by $\mathbb{C}\mathrm{ov}(S(\textbf{s}),S(\textbf{t}))=\sum_{j\in\mathcal{J}}j^2\sum_{k=1}^{M}\lambda_k^{(j)}\min\{s_k,t_k\}$.
\end{remark}
\begin{remark}
	Let $\{\boldsymbol{\Lambda}_k\}_{j\in\mathcal{J}}\subset\mathbb{R}^M_+$ be a finite sequence such that $\boldsymbol{\Lambda}_j=(\lambda_1^{(j)},\dots,\lambda_M^{(j)})\succ\textbf{0}$ for each $j\in\mathcal{J}$. Let $S_k=\{S_k(t),\ t\ge0\}$, $k=1,\dots,M$ be independent integer valued point processes. Then, in view of Theorem 3.1 of \cite{Vishwakarma2025}, the multiparameter process $S=\{S_1(t_1)+\dots+S_M(t_M),\ (t_1,\dots,t_M)\in\mathbb{R}^M_+\}$ is a GMSP, that is, $S\sim GMSP\{\boldsymbol{\Lambda}_j\}_{j\in\mathcal{J}}$ if and only if $S_k$ is GSP, that is, $S_k\sim GSP\{\lambda_k^{(j)}\}_{j\in\mathcal{J}}$ for each $k=1,\dots,M$.
\end{remark}
The following result provides a convergence result for GMSP. Its proof is similar to the case of one parameter generalized Skellam process studied in \cite{Cinque2025}. However, for one parameter case, we have a stronger weak convergence result.
\begin{proposition}
	For $k=1,\dots,M$, let $p^{(n)}_{l_k,j}\in(0,1)$ for all $l_k\ge1$, $n\ge1$ and $j\in\mathcal{J}$ such that $\sum_{j\in\mathcal{J}}p^{(n)}_{l_k,j}<1$. Let  $\{X^{(n)}_{l_k}\}_{l_1\ge1}$, $k=1,\dots,M$ be independent sequences of independent random variables with
	\begin{equation*}
		X_{l_k}^{(n)}=\begin{cases}
			j\in\mathcal{J},\ \text{with probability}\ p_{l_k,j}^{(n)},\\
			0,\ \text{with probability}\ 1-\sum_{j\in\mathcal{J}}p^{(n)}_{l_k,j}.
		\end{cases}
	\end{equation*}
Also, let us consider a $M$-parameter process $\{S^{(n)}(\textbf{t}),\ \textbf{t}\in\mathbb{R}^M_+\}$ defined as
	\begin{equation*}
		S^{(n)}(\textbf{t})\coloneqq\sum_{k=1}^{M}\sum_{l_k=1}^{[nt_k]}X_{l_k}^{(n)},\ \textbf{t}=(t_1,\dots,t_M)\in\mathbb{R}^M_+.
	\end{equation*}
	If 
	\begin{equation}\label{limconds}
		\sum_{k=1}^{M}\sum_{l_k=1}^{[nt_k]}p_{l_k,j}^{(n)}\longrightarrow\Lambda_j\cdot\textbf{t},\ \textbf{t}\in\mathbb{R}^M_+\ \text{and}\ \max_{0<l_k\leq n}p_{l_k,j}^{(n)}\longrightarrow0\ \text{as}\ n\longrightarrow\infty\ \text{for all $1\leq k\leq M$},
	\end{equation}
	then for $\textbf{t}^{(r)}$, $r=1,\dots,m$ in $\mathbb{R}^M_+$ such that $\textbf{0}\preceq\textbf{t}^{(1)}\preceq\dots\preceq\textbf{t}^{(m)}$, we have
	\begin{equation}\label{cnv1}
		(S^{(n)}(\textbf{t}^{(1)}),\dots,S^{(n)}(\textbf{t}^{(m)}))\overset{d}{\longrightarrow}(S(\textbf{t}^{(1)}),\dots,S(\textbf{t}^{(m)}))\ \text{as}\ n\longrightarrow\infty,
	\end{equation}
	where $\overset{d}{\longrightarrow}$ denotes the convergence in distribution, and $\{S(\textbf{t}),\ \textbf{t}\in\mathbb{R}^M_+\}$ is the GMSP. 
\end{proposition}
\begin{proof}
	For $\textbf{s}$ and $\textbf{t}$ in $\mathbb{R}^M_+$ such that $\textbf{s}\preceq\textbf{t}$, the increment of $\{S^{(n)}(\textbf{t}),\ \textbf{t}\in\mathbb{R}^M_+\}$ is given by 
	\begin{equation*}
		S^{(n)}(\textbf{t})-S^{(n)}(\textbf{s})=\sum_{k=1}^{M}\sum_{l_k=[ns_k]+1}^{[nt_k]}X_{l_k}^{(n)}.
	\end{equation*}
	Thus, $\{S^{(n)}(\textbf{t}),\ \textbf{t}\in\mathbb{R}^M_+\}$ has independent increments. So, to prove (\ref{cnv1}), it is enough to show that $S^{(n)}(\textbf{t})-S^{(n)}(\textbf{s})\overset{d}{\longrightarrow}S(\textbf{t})-S(\textbf{s})$ whenever $\textbf{s}\preceq\textbf{t}$. The characteristic function of $S^{(n)}(\textbf{t})-S^{(n)}(\textbf{s})$ is given by
	\begin{align*}
		\mathbb{E}\exp\big(iu(S^{(n)}(\textbf{t})-S^{(n)}(\textbf{s}))\big)&=\prod_{k=1}^{M}\prod_{l_k=[ns_k]+1}^{[nt_k]}\mathbb{E}e^{iuX_{l_k}^{(n)}}\\
		&=\prod_{k=1}^{M}\prod_{l_k=[ns_k]+1}^{[nt_k]}\bigg(\sum_{j\in\mathcal{J}}p^{(n)}_{l_k,j}(e^{iuj}-1)+1\bigg)\\
		&=\exp\bigg(\sum_{k=1}^{M}\sum_{l_k=[ns_k]+1}^{[nt_k]}\ln\bigg(\sum_{j\in\mathcal{J}}p^{(n)}_{l_k,j}(e^{iuj}-1)+1\bigg)\bigg)\\
		&\sim\exp\bigg(\sum_{k=1}^{M}\sum_{l_k=[ns_k]+1}^{[nt_k]}\sum_{j\in\mathcal{J}}p^{(n)}_{l_k,j}(e^{iuj}-1)\bigg),\ u\in\mathbb{R},
	\end{align*}
	where the last step follows from the second assumption in (\ref{limconds}) and using the approximation $\ln(1+x)\sim x$ as $x\to0$. Now, by using the first assumption in (\ref{limconds}), we get
	\begin{equation}\label{GMSPch}
		\lim_{n\longrightarrow\infty}\mathbb{E}\exp\big(iu(S^{(n)}(\textbf{t})-S^{(n)}(\textbf{s}))\big)=\exp\bigg(\sum_{j\in\mathcal{J}}\boldsymbol{\Lambda}_j\cdot(\textbf{t}-\textbf{s})(e^{iuj}-1)\bigg),\ u\in\mathbb{R}.
	\end{equation}
	As $S(\textbf{t})-S(\textbf{s})\overset{d}{=}S(\textbf{t}-\textbf{s})$ whenever $\textbf{s}\preceq\textbf{t}$, from (\ref{gmsppgf}), we have $\mathbb{E}e^{iuS(\textbf{t}-\textbf{s})}=\exp(\sum_{j\in\mathcal{J}}\boldsymbol{\Lambda}_j(\textbf{t}-\textbf{s})(e^{iuj}-1))$, $u\in\mathbb{R}$,
	which coincides with (\ref{GMSPch}). This completes the proof.
\end{proof}
\begin{proposition}\label{gsprep}
	Let $\{Y^k_l\}_{l\ge0}$, $k=1,\dots,M$ be sequences of independent and identically distributed (iid) random variables such that $\mathrm{Pr}\{Y^k_1=j\}=\lambda^{(j)}_k/\sum_{j\in\mathcal{J}}\lambda^{(j)}_k$, $j\in\mathcal{J}$, where $\lambda^{(j)}_k$'s are positive constants. Let $\{N_1(t),\ t\ge0\},\dots,\{N_M(t),\ t\ge0\}$ be one parameter Poisson processes with rates $\sum_{j\in\mathcal{J}}\lambda_1^{(j)},\dots,\sum_{j\in\mathcal{J}}\lambda_M^{(j)}$, respectively. Also, we assume that all the collections of random variables and processes appearing here are mutually independent. Then, the GMSP defined in (\ref{gmsp}) satisfies the following equality:
	\begin{equation*}
		S(\textbf{t})\overset{d}{=}\sum_{k=1}^{M}\sum_{l=1}^{N_k(t_k)}Y_l^k,\ \textbf{t}=(t_1,\dots,t_M)\in\mathbb{R}^M_+.
	\end{equation*}
\end{proposition}
\begin{proof}
	In view of Remark \ref{rem21}, the proof follows from Proposition 3.2 of \cite{Cinque2025}.
\end{proof}

The following result provides a compound representation for GMSP in a particular case.
\begin{proposition}\label{cpprep}
	Let $S\sim GMSP\{\Lambda_j\}_{j\in J}$ be a GMSP such that $\textbf{0}\prec\Lambda_j=(\lambda^{(j)},\dots,\lambda^{(j)})\in\mathbb{R}^M$ for each $j\in\mathcal{J}$, and let $Y_1,Y_2,\dots$ be iid random variables with common distribution $\mathrm{Pr}\{Y_1=j\}=\lambda^{(j)}/\sum_{j\in\mathcal{J}}\lambda^{(j)}$. Also, let $\{N(\textbf{t}),\ \textbf{t}\in\mathbb{R}^M_+\}$ be a MPP with rate parameter $(\sum_{j\in\mathcal{J}}\lambda^{(j)},\dots,\sum_{j\in\mathcal{J}}\lambda^{(j)})\in\mathbb{R}^M$, and it is independent of $Y_l$'s. Then, the following equality holds:
	\begin{equation*}
		S(\textbf{t})\overset{d}{=}\sum_{l=1}^{N(\textbf{t})}Y_l,\ \textbf{t}\in\mathbb{R}^M_+.
	\end{equation*}
\end{proposition}
\begin{proof}
	The pgf of $Y_1$ is given by $\mathbb{E}u^{Y_1}=\sum_{j\in\mathcal{J}}\lambda^{(j)}u^j/\sum_{j\in\mathcal{J}}\lambda^{(j)}$. Hence,
	\begin{align*}
		\mathbb{E}u^{\sum_{l=1}^{N(\textbf{t})}Y_l}&=\mathbb{E}\bigg(\frac{\sum_{j\in\mathcal{J}}\lambda^{(j)}u^j}{\sum_{j\in\mathcal{J}}\lambda^{(j)}}\bigg)^{N(\textbf{t})}\\
		&=\exp\bigg(\sum_{j\in\mathcal{J}}\lambda^{(j)}(t_1+\dots+t_M)\bigg(\frac{\sum_{j\in\mathcal{J}}\lambda^{(j)}u^j}{\sum_{j\in\mathcal{J}}\lambda^{(j)}}-1\bigg)\bigg)\\
		&=\exp\bigg((t_1+\dots+t_M)\sum_{j\in\mathcal{J}}\lambda^{(j)}(u^j-1)\bigg),
	\end{align*}
	which coincides with (\ref{gmsppgf}) for $\textbf{0}\prec\boldsymbol{\Lambda}_j=(\lambda^{(j)},\dots,\lambda^{(j)})\in\mathbb{R}^M$, $j\in\mathcal{J}$. This completes the proof.
\end{proof}

\begin{remark}
	Let $Y_1,Y_2,Y_3,\dots$ be iid random variables and $\{N(\textbf{t}), \textbf{t}\in\mathbb{R}^M_+\}$ be a MPP that is independent of $Y_l$'s. We consider a multiparameter compound Poisson process $\{\mathcal{N}(\textbf{t}),\ \textbf{t}\in\mathbb{R}^M_+\}$ defined by
	$\mathcal{N}(\textbf{t})\coloneqq\sum_{l=1}^{N(\textbf{t})}Y_l$. Then, it is a multiparameter L\'evy process in the sense of \cite{Pedersen2003, Pedersen2004a}. Hence, from Proposition 2.7 of \cite{Iafrate2024}, it follows that the process $\{\mathcal{N}(\textbf{t}),\ \textbf{t}\in\mathbb{R}^M_+\}$ is equal in distribution to a sum of one parameter independent L\'evy processes. Indeed, it is equal in distribution to a sum of $M$-many independent one parameter compound Poisson processes. Thus, Proposition \ref{cpprep} can be alternatively proved using Proposition \ref{gsprep} and Proposition 2.3 of \cite{Cinque2025}.
\end{remark}

\subsection{Multiparameter Skellam process} For $\mathcal{J}=\{1,-1\}$, the GMSP reduces to a multiparameter process $\{\tilde{S}(\textbf{t}),\ \textbf{t}\in\mathbb{R}^M_+\}$ defined by
\begin{equation*}
	\tilde{S}(\textbf{t})\coloneqq N_1(\textbf{t})-N_2(\textbf{t}),
\end{equation*}
where $\{N_i(\textbf{t}),\ \textbf{t}\in\mathbb{R}^M_+\}$, $i=1,2$ are independent MPPs with rate parameters $\boldsymbol{\Lambda}_i$, $i=1,2$, respectively. We call it the multiparameter Skellam process (MSP).

By using the convolution argument, the distribution $p(n,\textbf{t})=\mathrm{Pr}\{\tilde{S}(\textbf{t})=n\}$, $n\in\mathbb{Z}$ of MSP is given by
\begin{equation*}
	p(n,\textbf{t})=e^{-(\boldsymbol{\Lambda}_1+\boldsymbol{\Lambda}_2)\cdot\textbf{t}}\bigg(\frac{\boldsymbol{\Lambda}_1\cdot\textbf{t}}{\boldsymbol{\Lambda}_2\cdot\textbf{t}}\bigg)^{n/2}I_{|n|}(2\sqrt{(\boldsymbol{\Lambda}_1\cdot\textbf{t})(\boldsymbol{\Lambda}_2\cdot\textbf{t})}),\ n\in\mathbb{Z},\ \textbf{t}\in\mathbb{R}^M_+,
\end{equation*}
where $I_{|n|}$ is the modified Bessel function defined as follows:
\begin{equation}\label{Besselfn}
	I_{|n|}(x)=\sum_{r=0}^{\infty}\frac{(x/2)^{2m+|n|}}{\Gamma(m+|n|+1)m!},\ x\in\mathbb{R}.
\end{equation}

For $M=1$, the MSP reduces to the classical Skellam process $\{\tilde{S}(t)=N_1(t)-N_2(t),\ t\ge0\}$ introduced and studied in \cite{Skellam1946}, where $\{N_i(t),\ t\ge0\}$, $i=1,2$ are independent one parameter Poisson processes with rates $\lambda_i>0$, $i=1,2$, respectively. Its distribution is given by
\begin{equation*}
	\mathrm{Pr}\{\tilde{S}(t)=n\}=e^{-(\lambda_1+\lambda_2)t}\bigg(\frac{\lambda_1}{\lambda_2}\bigg)^{n/2}I_{|n|}(2\sqrt{\lambda_1\lambda_2}t),\ n\in\mathbb{Z},\ t\ge0.
\end{equation*} 

\subsection{Integral of GMSP} The path integrals of Markov processes find applications in variety of fields, such as in the study of genetic population (see \cite{Dhillon2025}), service and parking management system (see \cite{Vishwakarma2024a}), \textit{etc.}. In recent times, the Riemann-Liouville fractional integrals of one parameter point processes have been studied by many authors (see \cite{Dhillon2025, Kataria2025, Orsingher2013, Vishwakarma2024b, Vishwakarma2025b}, and reference therein). The integral of one parameter generalized Skellam process is introduced and studied in \cite{Cinque2025}. Moreover, some integrals of multiparameter Poisson processes are discussed in \cite{Vishwakarma2025}.

 Let $\{N(\textbf{t}),\ \textbf{t}\in\mathbb{R}^M_+\}$ be a MPP with rate parameter $\textbf{0}\prec\boldsymbol{\Lambda}=(\lambda_1,\dots,\lambda_M)\in\mathbb{R}^M$. Its Riemann integral over rectangle $\prod_{k=1}^{M}[0,t_k]$ in $\mathbb{R}^M_+$ is defined as follows:
\begin{equation}\label{intmpp}
	\mathscr{N}(\textbf{t})\coloneqq\int_{0}^{t_1}\dots\int_{0}^{t_M}N(s_1,\dots,s_M)\,\mathrm{d}s_1\dots\mathrm{d}s_M,\ \textbf{t}=(t_1,\dots,t_M)\in\mathbb{R}^M_+.
\end{equation}
 Its mean is given by $\mathbb{E}\mathscr{N}(\textbf{t})=\sum_{k=1}^{M}\lambda_kt_k^2\prod_{l\ne k}t_l/2$, $\textbf{t}=(t_1,\dots,t_M)\in\mathbb{R}^M_+$.

In the following result, we obtain an explicit expression for the characteristic function of the integral (\ref{intmpp}). For an integral of one parameter L\'evy process a similar result is obtained in \cite{Xia2018}. 
\begin{theorem}\label{propint}
	The characteristic function of integral (\ref{intmpp}) is given by
	\begin{equation}\label{mppint}
		\mathbb{E}\exp(iu\mathscr{N}(\textbf{t}))=\exp\bigg(\sum_{k=1}^{M}t_{k}\lambda_{k}\int_{0}^{1}\bigg(\exp\bigg(iu\prod_{k=1}^{M}t_kx\bigg)-1\bigg)\,\mathrm{d}x\bigg),\ u\in\mathbb{R},\ \textbf{t}=(t_1,\dots,t_k)\in\mathbb{R}^M_+.
	\end{equation}
\end{theorem}
\begin{proof}
For $\textbf{t}=(t_1,\dots,t_M)\in\mathbb{R}^M_+$, we have
\begin{align*}
	\int_{0}^{t_1}\dots&\int_{0}^{t_M}N(s_1,\dots,s_M)\,\mathrm{d}s_1\dots\mathrm{d}s_M\\
	&=\lim_{r_1\longrightarrow\infty}\dots\lim_{r_M\longrightarrow\infty}\bigg(\prod_{k=1}^{M}\frac{t_k}{r_k}\bigg)\sum_{l_1=1}^{r_1}\dots\sum_{l_M}^{r_M}N(t_1l_1/r_1,\dots,t_Ml_M/r_M)\\
	&=\lim_{r_1\longrightarrow\infty}\dots\lim_{r_M\longrightarrow\infty}\bigg(\prod_{k=1}^{M}\frac{t_k}{r_k}\bigg)\sum_{l_1=1}^{r_1}\dots\sum_{l_M}^{r_M}\{(r_1-l_1+1)(N(t_1l_1/r_1,\dots,t_Ml_M/r_M)\\
	&\hspace{3cm}-N(t_1(l_1-1)/r_1,\dots,t_Ml_M/r_M))+N(0,t_2l_2/r_2,\dots,t_Ml_M/r_M)\}\\
	&\overset{d}{=}\lim_{r_1\longrightarrow\infty}\dots\lim_{r_M\longrightarrow\infty}\bigg(\prod_{k=1}^{M}\frac{t_k}{r_k}\bigg)\sum_{l_1=1}^{r_1}\dots\sum_{l_M}^{r_M}(r_1-l_1+1)N(t_1/r_1,0,\dots,0)\\
	&\ \ +t_1\lim_{r_2\longrightarrow\infty}\dots\lim_{r_M\longrightarrow\infty}\bigg(\prod_{k=2}^{M}\frac{t_k}{r_k}\bigg)\sum_{l_2=1}^{r_2}\dots\sum_{l_M}^{r_M}N(0,t_2l_2/r_2,\dots,t_Ml_M/r_M)\\
	&\overset{d}{=}\prod_{k=2}^{M}t_k\lim_{r_1\longrightarrow\infty}\frac{t_1}{r_1}\sum_{l_1=1}^{r_1}l_1N(t_1/r_1,0,\dots,0)\\
	&\ \ +t_1\lim_{r_2\longrightarrow\infty}\dots\lim_{r_M\longrightarrow\infty}\bigg(\prod_{k=2}^{M}\frac{t_k}{r_k}\bigg)\sum_{l_2=1}^{r_2}\dots\sum_{l_M}^{r_M}(r_2-l_2+1)N(0,t_2/r_2,0,\dots,0)\\
	&\ \ +t_1t_2\lim_{r_3\longrightarrow\infty}\dots\lim_{r_M\longrightarrow\infty}\bigg(\prod_{3=2}^{M}\frac{t_k}{r_k}\bigg)\sum_{l_3=1}^{r_2}\dots\sum_{l_M}^{r_M}N(0,0,t_3l_3/r_3,\dots,t_Ml_M/r_M)\\
	&\overset{d}{=}\prod_{k=2}^{M}t_k\lim_{r_1\longrightarrow\infty}\frac{t_1}{r_1}\sum_{l_1=1}^{r_1}l_1N(t_1/r_1,0,\dots,0)+\prod_{k=1, k\ne2}^{M}t_k\lim_{r_2\longrightarrow\infty}\frac{t_2}{r_2}\sum_{l_2=1}^{r_2}l_2N(0,t_2/r_2,0,\dots,0)\\
	&\ \ +t_1t_2\lim_{r_3\longrightarrow\infty}\dots\lim_{r_M\longrightarrow\infty}\bigg(\prod_{3=2}^{M}\frac{t_k}{r_k}\bigg)\sum_{l_3=1}^{r_2}\dots\sum_{l_M}^{r_M}N(0,0,t_3l_3/r_3,\dots,t_Ml_M/r_M)\\
	&\ \ \vdots\\
	&\overset{d}{=}\sum_{k'=1}^{M}\prod_{k=1,k\ne k'}^{M}t_k\lim_{k'\longrightarrow\infty}\frac{t_{k'}}{r_{k'}}\sum_{l_{k'}=1}^{r_{k'}}l_{k'}N(0,\dots,0,t_{k'}/r_{k'},0,\dots,0),
\end{align*}
where we have used independent and stationary increments properties of MPP. Also, we note that $N(0,\dots,0,t_{k'}/r_{k'},0,\dots,0)$ are independent here because these are appearing due to independent increments of MPP. Hence,
\begin{align*}
	\mathbb{E}\exp\bigg(iu\int_{0}^{t_1}\dots&\int_{0}^{t_M}N(s_1,\dots,s_M)\,\mathrm{d}s_1\dots\mathrm{d}s_M\bigg)\\
	&=\prod_{k'=1}^{M}\lim_{k'\longrightarrow\infty}\prod_{l_{k'}=1}^{r_{k'}}\mathbb{E}\exp\bigg(iu\prod_{k=1}^{M}t_kl_{k'}/r_{k'}N(0,\dots,0,t_{k'}/r_{k'},0,\dots,0)\bigg),\ u\in\mathbb{R},\\
	&=\prod_{k'=1}^{M}\exp\bigg(\lim_{k'\longrightarrow\infty}\sum_{l_{k'}=1}^{r_{k'}}\frac{t_{k'}}{r_{k'}}\ln\mathbb{E}\exp\bigg(iu\prod_{k=1}^{M}t_kl_{k'}/r_{k'}N^{(k')}(0,\dots,0,1,0,\dots,0)\bigg)\bigg)\\
	&=\prod_{k'=1}^{M}\exp\bigg(\lim_{k'\longrightarrow\infty}\sum_{l_{k'}=1}^{r_{k'}}\frac{t_{k'}}{r_{k'}}\ln\exp\bigg(\lambda_{k'}\bigg(\exp\bigg(iu\prod_{k=1}^{M}t_kl_{k'}/r_{k'}\bigg)-1\bigg)\bigg)\bigg)\\
	&=\prod_{k'=1}^{M}\exp\bigg(t_{k'}\lambda_{k'}\int_{0}^{1}\bigg(\exp\bigg(iu\prod_{k=1}^{M}t_kx\bigg)-1\bigg)\,\mathrm{d}x\bigg),
\end{align*}
where $N^{(k')}(0,\dots,0,1,0,\dots,0)$ is a Poisson random variable with rate $\lambda_{k'}>0$ for each $k'=1,\dots,M$.
This completes the proof.
\end{proof}
\begin{remark}\label{remint}
	Note that a similar result, as in (\ref{mppint}), can be obtained for general multiparameter L\'evy process. Let $\{Y(\textbf{t}),\ \textbf{t}\in\mathbb{R}^M_+\}$ be a multiparameter L\'evy process in the sense of \cite{Pedersen2003, Pedersen2004a}. Then, from Proposition 2.7 of \cite{Iafrate2024}, it follows that there exist independent one parameter L\'evy processes $\{Y_k(t),\ t\ge0\}$, $k=1,\dots,M$, defined as $Y_k(t)\overset{d}{=}Y(t\textbf{e}_k)$ such that $Y(\textbf{t})\overset{d}{=}Y_1(t_1)+\dots+Y_M(t_M)$, $\textbf{t}=(t_1,\dots,t_M)\in\mathbb{R}^M_+$. Here, for each  $k=1,\dots,M$, $\textbf{e}_k$ is a  unit vector in $\mathbb{R}^M_+$ whose $k$th entry is $1$ and zero elsewhere. 
	
	Let $\phi_k$ denote the characteristic function of $Y_k(1)$ for each $k=1,\dots,M$. Then, following the proof of Theorem \ref{propint}, the characteristic function of the Riemann integral of $Y(\textbf{s})$ over rectangle $\prod_{k=1}^{M}[0,t_k]$ is given by
	\begin{equation}\label{lpint}
		\mathbb{E}\exp\bigg(iu\int_{0}^{t_1}\dots\int_{0}^{t_M}Y(s_1,\dots,s_M)\,\mathrm{d}s_1\dots\,\mathrm{d}s_M\bigg)=\exp\bigg(\sum_{k=1}^{M}t_k\int_{0}^{1}\ln\phi_k\bigg(u\prod_{k=1}^{M}t_kx\bigg)\,\mathrm{d}x\bigg),\ u\in\mathbb{R}.
	\end{equation}

For the case $M=1$, if $\{Y(t),\ t\ge0\}$ is a one parameter L\'evy process then the characteristic function of its Riemann integral is $\mathbb{E}\exp\big(iu\int_{0}^{t}Y(s)\,\mathrm{d}s\big)=\exp\big(t\int_{0}^{1}\ln\phi(utx)\,\mathrm{d}x\big)$, $u\in\mathbb{R}$, where $\phi$ is the characteristic function of $Y(1)$  (see \cite{Xia2018}). Thus, from (\ref{lpint}), we get
\begin{align*}
	\int_{0}^{t_1}\dots\int_{0}^{t_M}Y(s_1,\dots,s_M)\,\mathrm{d}s_1\dots\,\mathrm{d}s_M&\overset{d}{=}\sum_{k=1}^{M}\prod_{k'=1,k'\ne k}^{M}t_k\int_{0}^{t_k}Y_k(s_k)\,\mathrm{d}s_k\\
	&=\int_{0}^{t_1}\dots\int_{0}^{t_M}\sum_{k=1}^{M}Y_k(s_k)\,\mathrm{d}s_k,\ (t_1,\dots,t_M)\in\mathbb{R}^M_+.
\end{align*}
In particular, in case of MPP $Y_k$'s are independent one parameter Poisson processes with parameters $\lambda_k$'s. Hence, $\phi_k(u)=\exp(\lambda_k(e^{iu}-1))$, $u\in\mathbb{R}$ for each $k=1,\dots,M$.
\end{remark}

\begin{proposition}\label{intcmp}
	Let $\{\mathcal{N}(\textbf{t}),\ \textbf{t}\in\mathbb{R}^M_+\}$ be a multiparameter process defined by $\mathcal{N}(\textbf{t})\coloneqq\sum_{r=1}^{N(\textbf{t})}X_r$, where $X_1,X_2,\dots$ are iid random variables and $\{N(\textbf{t}),\ \textbf{t}\in\mathbb{R}^M_+\}$ is a MPP which is independent of $\{X_r\}_{r\ge1}$. Then, the following holds:
	\begin{equation}
		\int_{0}^{t_1}\dots\int_{0}^{t_M}\mathcal{N}(s_1,\dots,s_M)\,\mathrm{d}s_1,\dots,\,\mathrm{d}s_M\overset{d}{=}\prod_{k=1}^{M}t_k\sum_{r=1}^{N(\textbf{t})}X_rU_r,
	\end{equation}
	where $U_r$'s are iid $Uniform\,(0,1)$ that are independent of $X_r$'s and $\{N(\textbf{t}),\ \textbf{t}\in\mathbb{R}^M_+\}$.
\end{proposition}
\begin{proof}
	We note that $\{\mathcal{N}(\textbf{t}),\ \textbf{t}\in\mathbb{R}^M_+\}$ is  a multiparameter L\'evy process in the sense of \cite{Pedersen2003,Pedersen2004a}. Let $\{N_k(t),\ t\ge0\}$, $k=1,\dots,M$ be independent one parameter Poisson processes with rates $\lambda_k$, $k=1,\dots,M$, respectively. We assume that $N_k(t)$'s are independent of $X_r$'s. If $\psi(u)=\mathbb{E}e^{iuX_1}$, $u\in\mathbb{R}$, then by using (\ref{lpint}), we get
	\begin{align*}
		\mathbb{E}\exp\bigg(iu\int_{0}^{t_1}\dots\int_{0}^{t_M}\mathcal{N}(s_1,\dots,s_M)&\,\mathrm{d}s_1,\dots,\,\mathrm{d}s_M\bigg)\\
		&=\prod_{k=1}^{M}\exp\bigg(t_k\int_{0}^{1}\ln\mathbb{E}\exp\bigg(iu\prod_{k=1}^{M}t_kx\sum_{r=1}^{N_k(1)}X_r\bigg)\,\mathrm{d}x\bigg)\\
		&=\prod_{k=1}^{M}\exp\bigg(t_k\int_{0}^{1}\ln\mathbb{E}\psi\bigg(u\prod_{k=1}^{M}t_kx\bigg)^{N_k(1)}\,\mathrm{d}x\bigg)\\
		&=\prod_{k=1}^{M}\exp\bigg(\lambda_kt_k\int_{0}^{1}\bigg(\psi\bigg(u\prod_{k=1}^{M}t_kx\bigg)-1\bigg)\,\mathrm{d}x\bigg)\\
		&=\prod_{k=1}^{M}\exp\bigg(\lambda_kt_k\mathbb{E}\exp\bigg(iu\prod_{k=1}^{M}t_kX_1U_1\bigg)\bigg),\ u\in\mathbb{R}.
	\end{align*}
	This completes the proof.
\end{proof}

Let $S=\{S(\textbf{t}),\ \textbf{t}\in\mathbb{R}^M_+\}$ be the GMSP, that is, $S\sim GMSP\{\boldsymbol{\Lambda}_j\}_{j\in\mathcal{J}}$, where $\textbf{0}\prec\boldsymbol{\Lambda}_j=(\lambda^{(j)}_1,\dots,\lambda^{(j)}_M)$ for each $j\in\mathcal{J}$. Then, in view of Remark \ref{rem21} and Remark \ref{remint}, its Riemann integral over rectangle $\prod_{k=1}^{M}[0,t_k]$ satisfies the following equality:
 \begin{equation*}
	\int_{0}^{t_1}\dots\int_{0}^{t_M}S(s_1,\dots,s_M)\,\mathrm{d}s_1,\dots,\mathrm{d}s_M\overset{d}{=}\sum_{k=1}^{M}\prod_{k'=1,k'\ne k}^{M}t_k\int_{0}^{t_k}S_k(s_k)\,\mathrm{d}s_k\\,\ \textbf{t}=(t_1,\dots,t_M)\in\mathbb{R}^M_+,
\end{equation*}
where $S_k=\{S_k(t),\ t\ge0\}$ are independent GSP, that is, $S_k\sim GSP\{\lambda^{(j)}_k\}_{j\in\mathcal{J}}$, $k=1,\dots,M$. In \cite{Cinque2025}, p. 20, it is shown that $\int_{0}^{t}S_k(s)\,\mathrm{d}s\overset{d}{=}t\sum_{r=1}^{N_k(t)}X_r^{(k)}U_r$, where $X_1^{(k)}, X_2^{(k)},\dots$ are iid random variables with common probability mass function $\mathrm{Pr}\{X_1^{(k)}=j\}=\lambda^{(j)}_k/\sum_{j\in\mathcal{J}}\lambda_{k}^{(j)}$, $j\in\mathcal{J}$, $\{N_k(t),\ t\ge0\}$ is a Poisson process with parameter $\sum_{j\in\mathcal{J}}\lambda_k^{(j)}$, and $U_r$'s are independent $Uniform\,(0,1)$ random variables. Also, all the processes are independent of each other. Therefore,
\begin{equation}\label{f3}
	\int_{0}^{t_1}\dots\int_{0}^{t_M}S(s_1,\dots,s_M)\,\mathrm{d}s_1,\dots,\mathrm{d}s_M\overset{d}{=}\prod_{k=1}^{M}t_k\sum_{k=1}^{M}\sum_{r_k=1}^{N_k(t_k)}X_{r_k}^{(k)}U_{r_k}\\,\ \textbf{t}=(t_1,\dots,t_M)\in\mathbb{R}^M_+.
\end{equation}
In particular, if $\lambda_k^{(j)}=\lambda^{(j)}>0$ for each $k=1,\dots,M$, then from Proposition \ref{cpprep} and Proposition \ref{intcmp}, it follows that 
\begin{equation}\label{f2}
	\int_{0}^{t_1}\dots\int_{0}^{t_M}S(s_1,\dots,s_M)\,\mathrm{d}s_1,\dots,\mathrm{d}s_M\overset{d}{=}\prod_{k=1}^{M}t_k\sum_{r=1}^{N(\textbf{t})}X_rU_r,
\end{equation}
where $\{N(\textbf{t}),\ \textbf{t}\in\mathbb{R}^M_+\}$ is a MPP with rate parameter $\textbf{0}\prec\boldsymbol{\Lambda}=(\sum_{j\in\mathcal{J}}\lambda^{(j)},\dots,\sum_{j\in\mathcal{J}}\lambda^{(j)})\in\mathbb{R}^M$, $\{X_r\}_{r\ge1}$ is a sequence of iid random variables with common distribution $\mathrm{Pr}\{X_1=j\}=\lambda^{(j)}/\sum_{j\in\mathcal{J}}\lambda^{(j)}$, $j\in\mathcal{J}$, and $\{U_r\}_{r\ge1}$ is sequence of iid $Uniform\,(0,1)$ random variables. Also, all these collections are assume to be mutually independent. 
\begin{remark}
	Let $\{N_k(t),\ t\ge0\}$, $k=1,\dots,M$ be one parameter Poisson processes with rates $\lambda_k$, $k=1,\dots,M$, respectively. Let $\{X_r^{(k)}\}_{r\ge1}$, $k=1,\dots,M$ be identical sequences of iid random variables. It is assumed that all these collections are mutually independent. Let $\{X_r\}_{r\ge1}$ be a sequence of iid random variables such that $X_1\overset{d}{=}X_1^{(1)}$, and it is independent of a MPP $\{N(\textbf{t}),\ \textbf{t}\in\mathbb{R}^M_+\}$ with rate parameter $\textbf{0}\prec(\lambda_1,\dots,\lambda_k)$. 	If $\psi(u)=\mathbb{E}e^{iuX_1}$, $u\in\mathbb{R}$ is the characteristic function of $X_1$ then 
	\begin{equation*}
		\mathbb{E}\exp\bigg(iu\sum_{k=1}^{M}\sum_{r_k=1}^{N_k(t_k)}X^{(k)}_{r_k}\bigg)=\mathbb{E}(\phi(u))^{\sum_{k=1}^{M}N_k(t_k)}=\mathbb{E}(\phi(u))^{N(\textbf{t})},\ u\in\mathbb{R},\ \textbf{t}=(t_1,\dots,t_M)\in\mathbb{R}^M_+,
	\end{equation*} 
	as $N(\textbf{t})\overset{d}{=}\sum_{k=1}^{M}N_k(t_k)$.
	Thus, we get the following equality:
$\sum_{k=1}^{M}\sum_{r_k=1}^{N_k(t_k)}X^{(k)}_{r_k}\overset{d}{=}\sum_{r=1}^{N(\textbf{t})}X_r$. As a consequence of this, the equality (\ref{f2}) can be established alternatively by using (\ref{f3}).
\end{remark}

\section{A different version of multiparameter Skellam process}  For a finite subset $\mathcal{J}\in\mathbb{R}-\{0\}$, let $\{N_j(t),\ t\ge0\}$, $j\in\mathcal{J}$ be independent one parameter Poisson processes with rates $\lambda_j$, $j\in\mathcal{J}$, respectively. Let us consider a multiparameter process $\mathcal{S}=\{\mathcal{S}(\textbf{t}),\ \textbf{t}\in\mathbb{R}^{\#\mathcal{J}}_+\}$ defined as follows:
\begin{equation}\label{msp2}
	\mathcal{S}(\textbf{t})\coloneqq\sum_{j\in\mathcal{J}}jN_j(t_j),\ \textbf{t}=(t_j,\ j\in\mathcal{J})\in\mathbb{R}^{\#\mathcal{J}}_+,
\end{equation}
where $\#\mathcal{J}$ denotes the cordinality of set $\mathcal{J}$.  For $\textbf{s}=(s_j,\ j\in\mathcal{J})$ and $\textbf{t}=(t_j,\ j\in\mathcal{J})$ in $\mathbb{R}^{\#\mathcal{J}}_+$, its mean, variance and auto covariance are given by $\mathbb{E}\mathcal{S}(\textbf{t})=\sum_{j\in\mathcal{J}}j\lambda_jt_j$, $\mathbb{V}\mathrm{ar}\mathcal{S}(\textbf{t})=\sum_{j\in\mathcal{J}}j^2\lambda_jt_j$ and $\mathbb{C}\mathrm{ov}(\mathcal{S}(\textbf{s}),\mathcal{S}(\textbf{t}))=\sum_{j\in\mathcal{J}}j^2\min\{s_j,t_j\}$, respectively. Also, its pgf is given by
\begin{equation*}
	\mathbb{E}u^{\mathcal{S}(\textbf{t})}=\prod_{j\in\mathcal{J}}\mathbb{E}u^{jN_j(t_j)}=\exp\bigg(\sum_{j\in\mathcal{J}}\lambda_jt_j(u^j-1)\bigg),\ \textbf{t}=(t_j,\ j\in\mathcal{J}),\ 0<u\leq1.
\end{equation*}

When $t_j=t$ for all $j\in\mathcal{J}$, the multiparameter process defined by (\ref{msp2}) reduces to the one parameter generalized Skellam process introduced and studied in \cite{Cinque2025}.
\begin{remark}
	For $\textbf{s}=(s_j,\ j\in\mathcal{J})$ and $\textbf{t}=(t_j,\ j\in\mathcal{J})$ in $\mathbb{R}^{\#\mathcal{J}}_+$ such that $\textbf{s}\preceq\textbf{t}$, the increment of \{$\mathcal{S}(\textbf{t}), \textbf{t}\in\mathbb{R}^{\#\mathcal{J}}_+\}$ is given by $\mathcal{S}(\textbf{t})-\mathcal{S}(\textbf{s})=\sum_{j\in\mathcal{J}}j(N_j(t_j)-N_j(s_j))$. Thus, by using the independence of $\{N_j(t),\ t\ge0\}$'s and their stationary and independent increments properties, the \{$\mathcal{S}(\textbf{t}), \textbf{t}\in\mathbb{R}^{\#\mathcal{J}}_+\}$ has stationary and independent increments. The characteristic function of its increment is given by
	\begin{equation}\label{diffcf3}
		\mathbb{E}\exp(iz(\mathcal{S}(\textbf{t})-\mathcal{S}(\textbf{s})))=\prod_{j\in\mathcal{J}}\mathbb{E}\exp(izjN_j(t_j-s_j))=\exp\bigg(\sum_{j\in\mathcal{J}}\lambda_j(t_j-s_j)(e^{izj}-1)\bigg),\ z\in\mathbb{R}.
	\end{equation}
\end{remark}
\begin{theorem}\label{thm31}
	Let $\{\alpha_k\}_{k\ge1}$ be a sequence of non-negative real numbers such that $\alpha_k\longrightarrow\infty$ as $k\longrightarrow\infty$. For $j\in\mathcal{J}$, let $p_{l_j,j'}^{(k)}\in(0,1)$ for all $l_j\ge1$, $k\ge1$, $j'\in\mathcal{J}$ such that $\sum_{j'\in\mathcal{J}}p_{l_j,j'}^{(k)}<1$, and  $\{X_{l_j}^{(k)}\}_{l_j\ge1}$ be independent sequences of independent random variables such that
	\begin{equation}\label{xdef}
		X_{l_j}^{(k)}=\begin{cases}
		 j'\in\mathcal{J},\ \text{with probability $p_{l_j,j'}^{(k)}$},\\
		 0,\ \text{with probability 1-$\sum_{j'\in\mathcal{J}}p_{l_j,j'}^{(k)}$}.
		\end{cases}
	\end{equation}
	
	Let $\{U_k(\textbf{t}), \textbf{t}\in\mathbb{R}^{\#\mathcal{J}}_+\}$, $k\ge1$ be a multiparameter process defined as
	\begin{equation*}
		U_k(\textbf{t})=\sum_{j\in\mathcal{J}}\sum_{l_j=1}^{[\alpha_kt_j]}X_{l_j}^{(k)},\ \textbf{t}=(t_j,\ j\in\mathcal{J})\in\mathbb{R}^{\#\mathcal{J}}_+.
	\end{equation*}
	If for each $j\in\mathcal{J}$
	\begin{equation}\label{assm31}
		 \max_{0<l_j\leq\alpha_k}p_{l_j,j'}^{(k)}\longrightarrow0\ \text{for all $j'\in\mathcal{J}$}
	\end{equation}
	and for all $t\ge0$
	\begin{equation}\label{assm32}
		\sum_{l_j=1}^{[\alpha_kt]}p_{l_j,j'}^{(k)}\longrightarrow\begin{cases}
			\lambda_jt,\ j=j',\\
			0,\ j\ne j'
		\end{cases}
	\end{equation}
	as $k\longrightarrow\infty$,
	then for any points $\textbf{t}^1,\dots,\textbf{t}^m$ in $\mathbb{R}^{\#\mathcal{J}}_+$ such that $\textbf{t}^r\preceq\textbf{t}^{r+1}$, $r=1,\dots,m-1$, we have
	\begin{equation}\label{lim31}
		(U_k(\textbf{t}^1),\dots,U_k(\textbf{t}^m))\overset{d}{\longrightarrow}(\mathcal{S}(\textbf{t}^1),\dots,\mathcal{S}(\textbf{t}^m))\ \text{as}\ k\longrightarrow\infty.
	\end{equation}
\end{theorem}
\begin{proof}
	For $\textbf{t}=(t_j,\ j\in\mathcal{J})$ and $\textbf{s}=(s_j,\ j\in\mathcal{J})$ in $\mathbb{R}^{\#\mathcal{J}}_+$ such that $\textbf{s}\preceq\textbf{t}$, the increment of $U_k(\textbf{t})$ is given by
	\begin{equation}\label{diffu}
		U_k(\textbf{t})-U_k(\textbf{s})=\sum_{j\in\mathcal{J}}\sum_{l_j=[\alpha_ks_j]+1}^{[\alpha_kt_j]}X_{l_j}^{(k)}.
	\end{equation}
	Thus, $\{U_k(\textbf{t}),\ \textbf{t}\in\mathbb{R}^{\#\mathcal{J}}_+\}$ has independent increments. Hence, to prove (\ref{lim31}), it is enough prove the convergence $U_k(\textbf{t})-U_k(\textbf{s})\overset{d}{\longrightarrow}\mathcal{S}(\textbf{t})-\mathcal{S}(\textbf{s})$ as $k\longrightarrow\infty$ whenever $\textbf{s}\preceq\textbf{t}$.
	
	By using (\ref{xdef}), the characteristic function of (\ref{diffu}) is given by
	\begin{align*}
		\mathbb{E}\exp(iz(U_k(\textbf{t})-U_k(\textbf{s})))&=\prod_{j\in\mathcal{J}}\prod_{l_j=[\alpha_ks_j]+1}^{[\alpha_kt_j]}\mathbb{E}\exp(izX_{l_j}^{(k)}),\ z\in\mathbb{R},\\
		&=\prod_{j\in\mathcal{J}}\prod_{l_j=[\alpha_ks_j]+1}^{[\alpha_kt_j]}\bigg(\sum_{j'\in\mathcal{J}}e^{izj'}p_{l_j,j'}^{(k)}+1-\sum_{j'\in\mathcal{J}}p_{l_j,j'}^{(k)}\bigg)\\
		&=\prod_{j\in\mathcal{J}}\prod_{l_j=[\alpha_ks_j]+1}^{[\alpha_kt_j]}\exp\bigg(\ln\bigg(1+\sum_{j'\in\mathcal{J}}(e^{izj'}-1)p_{l_j,j'}^{(k)}\bigg)\bigg)\\
		&\sim\exp\bigg(\sum_{j\in\mathcal{J}}\sum_{l_j=[\alpha_ks_j]+1}^{[\alpha_k t_j]}\sum_{j'\in\mathcal{J}}(e^{izj'}-1)p_{l_j,j'}^{(k)}\bigg)\\
		&=\exp\bigg(\sum_{j'\in\mathcal{J}}(e^{izj'}-1)\sum_{j\in\mathcal{J}}\bigg(\sum_{l_j=1}^{[\alpha_kt_j]}p_{l_j,j'}^{(k)}-\sum_{l_j=1}^{[\alpha_ks_j]}p_{l_j,j'}^{(k)}\bigg)\bigg),
	\end{align*}
	where the approximation in penultimate step follows from (\ref{assm31}) and $\ln(1+x)\sim x$ as $x\longrightarrow0$. Now, on using (\ref{assm32}), we get
	\begin{equation*}
		\lim_{k\to\infty}	\mathbb{E}\exp(iz(U_k(\textbf{t})-U_k(\textbf{s})))=\exp\bigg(\sum_{j'\in\mathcal{J}}(e^{izj'}-1)\lambda_{j'}(t_{j'}-s_{j'})\bigg),\ z\in\mathbb{R},
	\end{equation*}
	which coincides with (\ref{diffcf3}). This completes the proof.
\end{proof}

\begin{remark}
	Let us consider a sequence $\{\beta_k\}_{k\ge1}$ such that $\beta_k>\sum_{j\in\mathcal{J}}\lambda_j$ for all $k\ge1$ and $\beta_k\longrightarrow\infty$ as $k\longrightarrow\infty$. Then, by taking  $p_{l_j,j'}^{(k)}=\lambda_j\delta_{j,j'}/\beta_k$, $(j,j')\in\mathcal{J}\times\mathcal{J}$ and $\beta_k=\alpha_k$, $k\ge1$, we note that the assumptions of Theorem \ref{thm31} satisfy. Here, $\delta_{j,j'}$ is the Kronecker delta function.
\end{remark}
\subsection{Fractional two parameter Skellam process} In \cite{Kerss2014}, a time-changed version of the Skellam process was studied, highlighting its potential relevance to finance. Skellam processes of order $k$, and their time-changed variants are discussed in \cite{Gupta2020, Kataria2024a}. For $\mathcal{J}=\{1,-1\}$, the multiparameter process defined in (\ref{msp2}) reduces to a two parameter Skellam process $\{\mathcal{S}(t_1,t_2),\ (t_1,t_2)\in\mathbb{R}^2_+\}$ defined as follows:
\begin{equation}\label{tsp}
	\mathcal{S}(t_1,t_2)\coloneqq N_1(t_1)-N_2(t_2),
\end{equation}
where $\{N_j(t),\ t\ge0\}$, $j=1,2$ are independent Poisson processes with rates $\lambda_1>0$ and $\lambda_2>0$, respectively. Its distribution is given by
\begin{equation*}
	\mathrm{Pr}\{\mathcal{S}(t_1,t_2)=n\}=e^{-\lambda_1t_1-\lambda_2t_2}\bigg(\frac{\lambda_1t_1}{\lambda_2t_2}\bigg)^{n/2}I_{|n|}(2\sqrt{\lambda_1\lambda_2t_1t_2}),\ n\in\mathbb{Z},\ (t_1,t_2)\in\mathbb{R}^2_+.
\end{equation*}
For $(s_1,s_2)$ and $(t_1,t_2)$ in $\mathbb{R}^2_+$, its mean, variance  and auto covariance are given by $\mathbb{E}\mathcal{S}(t_1,t_2)=\lambda_1t_1-\lambda_2t_2$, $\mathbb{V}\mathrm{ar}\mathcal{S}(t_1,t_2)=\lambda_1t_1+\lambda_2t_2$ and $\mathbb{C}\mathrm{ov}(\mathcal{S}(s_1,s_2),\mathcal{S}(t_1,t_2))=\lambda_1\min\{s_1,t_1\}+\lambda_2\min\{s_2,t_2\}$, respectively.

Next, we consider a time-changed two parameter Skellam process. First, we recall the definition of inverse stable subordinator (see \cite{Meerschaert2011}).

A non-negative real valued L\'evy process $\{D^\alpha(t),\ t\ge0\}$, $\alpha\in(0,1)$ with almost surely non-decreasing sample path is called $\alpha$-stable subordinator if its Laplace transform is given by $\mathbb{E}e^{-uD^\alpha(t)}=e^{-tu^\alpha}$, $u>0$.

 The first passage time process $\{L^\alpha(t),\ t\ge0\}$ defined by $L^\alpha(t)\coloneqq\inf\{s\ge0:D^\alpha(s)\ge t\}$ is called inverse $\alpha$-stable subordinator.
 
 For $\alpha,\beta\in(0,1)$, let $\{L_1^{\alpha}(t),\ t\ge0\}$ and $\{L_2^\beta(t),\ t\ge0\}$ be two independent inverse stable subordinators. We consider the following time-changed two parameter process:
\begin{equation}\label{ftsp}
	\mathcal{S}^{\alpha,\beta}(t_1,t_2)\coloneqq\mathcal{S}(L_1^\alpha(t_1),L_2^\beta(t_2))=N_1(L_1^\alpha(t_1))-N_2(L_2^\beta(t_2)),\ (t_1,t_2)\in\mathbb{R}^2_+,
\end{equation}
where $\{\mathcal{S}(t_1,t_2),\ (t_1,t_2)\in\mathbb{R}^2_+\}$ is a two parameter Skellam process as defined in (\ref{tsp}). It is assumed that all the processes appearing here are independent of each other. 

Recall that the Poisson process time-changed by an inverse stable subordinator is a fractional Poisson process in the sense of \cite{Meerschaert2011}. Hence, the process $\{\mathcal{S}^{\alpha,\beta}(t_1,t_2),\ (t_1,t_2)\in\mathbb{R}^2_+\}$ is a difference of two independent one parameter fractional Poisson processes. 

The distribution of $N_1(L_1^\alpha(t))$ is given by (see \cite{Meerschaert2011})
\begin{equation*}
	\mathrm{Pr}\{N_1(L_1^\alpha(t))=n\}=\frac{(\lambda_1 t^\alpha)^n}{n!}\sum_{r=0}^{\infty}\frac{(n+r)!}{r!}\frac{(-\lambda_1t^\alpha)^r}{\Gamma(\alpha(n+r)+1)},\ n\ge0,\ t\ge0.
\end{equation*}
Thus, for $n\ge0$, the point probabilities of $\mathcal{S}^{\alpha,\beta}(t_1,t_2)$ are
\begin{align*}
	\mathrm{Pr}\{\mathcal{S}^{\alpha,\beta}(t_1,t_2)=n\}&=\sum_{l=0}^{\infty}\mathrm{Pr}\{N_1{L_1^\alpha(t_1)}=n+l\}\mathrm{Pr}\{N_2(L_2^\beta(t_2))=l\}\\
	&=\sum_{l=0}^{\infty}\frac{(\lambda_1t_1^\alpha)^{n+l} (\lambda_2t_2^\beta)^l}{(n+l)!l!}\sum_{r_1=0}^{\infty}\sum_{r_2=0}^{\infty}\frac{(n+l+r_1)!(l+r_2)!(-\lambda_1t_1^\alpha)^{r_1}(-\lambda_2 t_2^\beta)^{r_2}}{\Gamma(\alpha(n+l+r_1)+1)\Gamma(\beta(l+r_2)+1)r_1!r_2!}\\
	&=(\lambda_1t_1^\alpha)^n\sum_{r_1=0}^{\infty}\sum_{r_2=0}^{\infty}\frac{(-\lambda_1t_1^\alpha)^{r_1}(-\lambda_2 t_2^\beta)^{r_2}}{r_1!r_2!}\\
	&\hspace{1cm}\cdot{}_2\Psi_3\Bigg[\begin{matrix}
		\ \ \ \ \ \ \ \ \ \ \ \ (n+r_1+1,1),&(r_2+1,1)\\\\
		(\alpha(n+r_1)+1,\alpha),&(\beta r_2+1,\beta),&(n+1,1)
	\end{matrix}\bigg|\lambda_1\lambda_2t_1^\alpha t_2^\beta\Bigg],
\end{align*}
where ${}_2\Psi_3$ is the generalized Wright function, for definition see \cite{Kilbas2006}, p. 56, Eq. (1.11.14).

For $n<0$, we have
\begin{align*}
	\mathrm{Pr}\{\mathcal{S}^{\alpha,\beta}(t_1,t_2)=n\}&=\sum_{l=0}^{\infty}\mathrm{Pr}\{N_1{L_1^\alpha(t_1)}=l\}\mathrm{Pr}\{N_2(L_2^\beta(t_2))=l+|n|\}\\
	&=(\lambda_2t_2^\beta)^{|n|}\sum_{r_1=0}^{\infty}\sum_{r_2=0}^{\infty}\frac{(-\lambda_1t_1^\alpha)^{r_1}(-\lambda_2 t_2^\beta)^{r_2}}{r_1!r_2!}\\
	&\hspace{1cm}\cdot{}_2\Psi_3\Bigg[\begin{matrix}
		\ \ \ \ \ \ \ \ \ \ \ \ (|n|+r_1+1,1),&(r_2+1,1)\\\\
		(\beta(|n|+r_1)+1,\beta),&(\alpha r_2+1,\alpha),&(|n|+1,1)
	\end{matrix}\bigg|\lambda_1\lambda_2t_1^\alpha t_2^\beta\Bigg].
\end{align*}

For $(t_1,t_2)\in\mathbb{R}^2_+$, the mean and variance of (\ref{ftsp}) are given by
\begin{equation*}
	\mathbb{E}\mathcal{S}^{\alpha,\beta}(t_1,t_2)=\frac{\lambda_1t_1^\alpha}{\Gamma(\alpha+1)}-\frac{\lambda_2t_2^\beta}{\Gamma(\beta+1)}
\end{equation*}
and
\begin{align*}
	\mathbb{V}\mathrm{ar}\mathcal{S}^{\alpha,\beta}(t_1,t_2)&=\frac{\lambda_1t_1^\alpha}{\Gamma(\alpha+1)}+\frac{\lambda_1t_1^\alpha}{\alpha}\bigg(\frac{1}{\Gamma(2\alpha)}-\frac{1}{\alpha\Gamma^2(\alpha)}\bigg)\\
	&\ \ +\frac{\lambda_2t_2^\beta}{\Gamma(\beta+1)}+\frac{\lambda_2t_2^\beta}{\beta}\bigg(\frac{1}{\Gamma(2\beta)}-\frac{1}{\beta\Gamma^2(\beta)}\bigg),
\end{align*}
respectively.


\begin{thebibliography}{00}
\bibitem{Barndorff-Nielsen2001}
Barndorff-Nielsen, O.E., Pedersen, J. and Sato, K., 2001. Multivariate subordination, self
decomposability and stability. \textit{Adv. Appl. Probab.} \textbf{33}(1), 160-187.
\bibitem{Barndorff-Nielsen2011}
Barndorff-Nielsen, O.E., Pollard, D., Shephard, N., 2011. Integer-valued L\'evy processes and low latency financial econometrics. \textit{Quant. Finance} \textbf{12}(4). 587-605.
\bibitem{Cao1999}
Cao, J., Worsley, K., 1999. The geometry of correlation fields with an application to functional connectivity of
the brain. \textit{Ann. Appl. Probab.} \textbf{9}(4), 1021-1057.
\bibitem{Carr2011}
Carr, P., 2011. Semi-static hedging of barrier options under Poisson jumps. \textit{Int. J. Theor. Appl. Finance} \textbf{14}, 1091-1111.
\bibitem{Cinque2025}
Cinque, F., Orsingher, E., 2025. Point processes of the Poisson-Skellam family.  arXiv:2504.07672v1
\bibitem{Dhillon2025}
Dhillon, M., Vishwakarma, P., Kataria, K.K., 2025. On a fractional variant of linear birth-death process.  arXiv:2502.07329v1 
\bibitem{Irwin1937}
Irwin, J.O., 1937. The frequency distribution of the difference between two independent variates following the same Poisson distribution. \textit{J. of the Royal Statistical Society, Ser. A} \textbf{100}, 415-416.
\bibitem{Gupta2020}
Gupta, N., Kumar, A., Leonenko, N., 2020. Skellam type processes of order $k$ and beyond.\textit{ Entropy} \textbf{22}(11), 1193.
\bibitem{Iafrate2024}
Iafrate, F., Ricciuti, C., 2024. Some families of random fields related to multiparameter Lévy processes. \textit{J. Theor. Probab.} \textbf{37}, 3055-3088.
\bibitem{Khoshnevisan2002}
Khoshnevisan, D., 2002. Multiparameter processes. In: An Introduction to Random Fields, Springer
Monographs in Mathematics.
\bibitem{Kilbas2006}
Kilbas, A.A., Srivastava, H.M., Trujillo, J.J., 2006. Theory and Applications of Fractional Differential Equa
tions. Elsevier Science B.V., Amsterdam.
\bibitem{Kerss2014}
Kerss, A., Leonenko, N.N., Sikorskii, A., 2014. Fractional Skellam processes with applications to finance. \textit{Frac. Cal. Appl. Anal.} \textbf{17}(2), 532-551.
\bibitem{Kataria2022}
Kataria, K.K., Khandakar, M., 2022.  Skellam and time-changed variants of the generalized fractional counting process. \textit{Frac. Cal. Appl. Anal.} \textbf{25}, 1873-1907.
\bibitem{Kataria2024a}
Kataria, K.K., Khandakar, M., 2024. Fractional Skellam process of order $k$. \textit{J. Theor. Probab.} \textbf{37}, 1333-1356.
\bibitem{Kataria2025}
Kataria, K.K., Vishwakarma, P., 2025. Fractional Poisson random fields on $\mathbb{R}^2_+$. \textit{J. Stat. Phys.} \textbf{192}(6), 82.
\bibitem{Meerschaert2011}
Meerschaert, M.M., Nane, E., Vellaisamy, P., 2011. The fractional Poisson process and the inverse stable subordinator. \textit{Electron. J. Probab.} \textbf{16}(59), 1600-1620.
\bibitem{Orsingher2013}
Orsingher, E., Polito, F., 2013. On the integral of fractional Poisson processes. \textit{Stat. Probab. Lett.} \textbf{83}, 1006-1017.
\bibitem{Pedersen2003}
Pedersen, J., Sato, K., 2003. Cone-parameter convolution semigroups and their subordination. \textit{Tokyo J. Math.}	\textbf{26}(2), 503–525.
\bibitem{Pedersen2004a}
Pedersen, J., Sato, K., 2004. Relations between cone-parameter L\'evy processes and convolution semigroups. \textit{J. Math. Soc. Japan}. \textbf{56}(2), 541–559.
\bibitem{Skellam1946}
Skellam, J.G., 1946. The frequency distribution of the difference between two Poisson variates
belonging to different populations. \textit{J. R. Stat. Soc. (N.S.)} \textbf{109}, 296.
\bibitem{Vishwakarma2025}
Vishwakarma, P., Kataria, K.K., 2025. Multiparameter Poisson processes and martingales.  arXiv:2501.09543v1
\bibitem{Vishwakarma2024a}
Vishwakarma, P., Kataria, K.K., 2024. On the generalized birth-death process and its linear versions. \textit{J. Theor.
	Probab.} \textbf{37}, 3540-3780.
\bibitem{Vishwakarma2024b}
Vishwakarma, P., Kataria, K.K., 2024. On integrals of birth-death processes at random time. \textit{Stat. Probab.	Lett.} \textbf{214}, 110204.
\bibitem{Vishwakarma2025b}
Vishwakarma, P., 2025. Fractional Skellam random fields on $\mathbb{R}^2_+$.  arXiv:2509.10870v1
\bibitem{Xia2018}
Xia, W., 2018. On the distribution of running average of Skellam process. \textit{Int. J. Pure Appl.
Math.} \textbf{119}, 461-473.
\end{thebibliography}
\end{document}